\documentclass{article}

\usepackage{amsmath}
\usepackage{amsfonts}
\usepackage{amsthm}
\usepackage{graphicx}
\usepackage[T1]{fontenc}
\usepackage[utf8]{inputenc}
\usepackage{url}

\newcommand{\OS}{Ozsv\'ath--Szab\'o\ }
\newcommand{\Z}{\mathbb{Z}}  
\newcommand{\R}{\mathbb{R}}  
\newcommand{\LL}{\mathbb{L}}  
\newcommand{\M}{\mathbb{M}}  
\newcommand{\Lx}{\mathbb{L}\!^\times}  
\renewcommand{\S}{\mathrm{S}}  
\newcommand{\T}{\mathrm{T}}  
\renewcommand{\over}{over }
\newcommand{\HFhat}{\widehat{HF}}
\newcommand{\HFp}{HF^+}
\newcommand{\SL}{\mathrm{SL}}  
\newcommand{\SFH}{\mathrm{SFH}}

\newcommand{\twSFH}{\underline{\mathrm{SFH}}}
\newcommand{\twc}{\underline{c}}
\newcommand{\Tor}{\mathrm{Tor}}

\newcommand{\twV}{\underline{V}}
\newcommand{\twPhi}{\underline{\Phi}}

\newtheorem*{main-thm}{Main theorem}
\newtheorem{thm}{Theorem}
\newtheorem{lem}[thm]{Lemma}
\newtheorem{prop}[thm]{Proposition}
\newtheorem{cor}[thm]{Corollary}
\newtheorem{ex}[thm]{Example}
\newtheorem{defn}[thm]{Definition}

\author{Patrick Massot}

\title{Infinitely many universally tight torsion free contact structures 
with vanishing \OS contact invariants}

\date{December 2009}

\begin{document}

\maketitle

\begin{abstract}
\OS contact invariants are a powerful way to prove tightness of contact
structures but they are known to vanish in the presence of Giroux torsion. In
this paper we construct, on infinitely many manifolds, infinitely many isotopy
classes of universally tight torsion free contact structures whose \OS
invariant vanishes. We also discuss the relation between these invariants
and an invariant on $\T^3$ and construct other
examples of new phenomena in Heegaard--Floer theory.
Along the way, we prove two conjectures of K~Honda, W~Kazez and G~Matić about
their contact topological quantum field theory. Almost all the proofs in this
paper rely on their gluing theorem for sutured contact invariants.

\end{abstract}

\section*{Introduction}

Contact topology studies isotopy classes of contact structures\footnote{all
manifolds in this paper are oriented and all contact structures are positive}.
These classes come in two main flavors: overtwisted and tight, the latter being
further divided into universally tight and virtually overtwisted. Up to now,
besides homotopical data, there are only two algebraic objects which have been
successfully used to classify such isotopy classes on a general 3--manifold. The
first one is Giroux torsion introduced in \cite{Giroux_2000}, its definition is
recalled in Section \ref{S:partitions}. It is either a non-negative integer or
infinite and always infinite for overtwisted classes. It is invariant under
isomorphisms, not only isotopies. It shares the monotonicity property of
symplectic capacities \cite{HZ} on one hand and the finiteness property of
3--manifolds complexity \cite{Matveev} on the other hand. Indeed, if $(M, \xi)
\subset (M', \xi')$ then $\Tor(\xi) \leq \Tor(\xi')$ and, for fixed $M$ and $n$,
there are only finitely many isomorphism classes of contact structures on $M$
whose torsion is at most $n$. Another way to put it is to say that finite
torsion determines contact structures up to isomorphism and a finite ambiguity.
More generally, it plays an important role in the coarse classification of tight
contact structures \cite{CGH}. The second object, based on open book
decompositions \cite{Giroux_Pekin}, is \OS contact invariants introduced in
\cite{OSz_contact} which live in the Heegaard--Floer homology of the ambient
manifold. They come in various flavors depending on a choice of coefficients.
These invariants are a powerful tool to detect tightness and obstructions to
fillability by symplectic or complex manifolds. Its main properties are listed
in Theorem \ref{thm:ci} below. 

It is natural to investigate relations between these two invariants. In
\cite{GHV}, P~Ghiggini, K~Honda and J~Van Horn Morris proved that, whenever
Giroux torsion is non zero, the contact invariant \over $\Z$ coefficients
vanishes (we give a new proof of this result in Section \ref{S:vanishing}). Here
we prove that the converse does not hold.

\begin{main-thm}[Section \ref{S:vanishing}]
Every Seifert manifold whose base has genus at least three supports infinitely
many (explicit) isotopy classes of universally tight torsion free contact
structures whose \OS invariant \over $\Z$ coefficients vanishes.
\end{main-thm}

In the above theorem, the genus hypothesis cannot be
completely dropped because, for instance, on the sphere $\S^3$ and the torus
$\T^3$, all torsion free contact structures have non vanishing \OS invariants.
However, it may hold for genus two bases. Note that the class of Seifert
manifolds is the only one where isotopy classes of contact structures are pretty
well understood. So the theorem says that examples of universally tight torsion
free contact structures with vanishing \OS invariant exist on all manifolds we
understand, provided there is enough topology (the base should have genus at
least three). In this statement, isotopy classes cannot be replaced by conjugacy
classes because of the finiteness property explained above. Along the way we
prove Conjecture 7.13 of \cite{HKM_tqft}.

Our examples also provide a corollary in the world of Legendrian knots. \OS
theory provides invariants for Legendrian or transverse knots in different
(related) ways, see \cite{SV_leg} and references therein. In the standard
contact 3--spheres there are still two seemingly distinct ways to define such
invariants but, in general contact manifolds, the known invariants all come from
the sutured contact invariant of the complement of the knot according to the
main theorem proved by V~Vértesi and A~Stipsicz in \cite{SV_leg}. In this
paper they call strongly non loose those Legendrian knots in overtwisted contact
manifolds whose complement is tight and torsion free.  Corollary 1.2 of that
papers states that a Legendrian knot has vanishing invariant when it is not
strongly non loose. We prove that the converse does not hold.

\begin{thm}[see the discussion after Proposition \ref{prop:main}]
	\label{thm:leg}
There exists, in an overtwisted contact manifold, a null-homologous strongly non
loose Legendrian knot whose sutured invariant vanishes (the construction is
explicit).
\end{thm}

After studying the relationship between \OS invariants and Giroux torsion, we
now turn to a more specific relation between these invariants and an invariant
defined only on the 3--torus. E~Giroux proved that any two incompressible
prelagrangian tori of a tight contact structure $\xi$ on $\T^3$ are isotopic. We
can then define the Giroux invariant $G(\xi) \in H_2(\T^3)/\pm 1$ to be the
homology class of its prelagrangian incompressible tori. Note that there is a
``sign ambiguity'' because these tori are not naturally oriented. Translated
into this language, Giroux proved that two tight contact structures on $\T^3$ are
isotopic if and only if they have the same Giroux invariant and the same Giroux
torsion, see \cite{Giroux_2000}. This invariant is clearly
$\text{Diff}(\T^3)$--equivariant. Since this group acts transitively on
primitive elements of $H_2(\T^3)$, we see that all these elements are attained by
$G$. This also proves that all tight contact structures on $\T^3$ which have the
same torsion are isomorphic. This classification of tight contact structures on
$\T^3$ and a result by Y~Eliashberg shows that torsion free contact structures on
$\T^3$ are exactly the Stein fillable ones.

\begin{thm}[Section \ref{S:t3}]
\label{thm:t3}
There is a unique up to sign $H_1(\T^3)$--equivariant isomorphism between
$\HFhat(\T^3)$ and $H^1(\T^3) \oplus H^2(\T^3)$ (on the ordinary cohomology side,
$H_1$ sends $H^1$ to zero and $H^2$ to $H^1$ by slant product). Under this
isomorphism, the \OS invariant of a torsion free contact structure on $\T^3$ is
sent to the Poincaré dual of its Giroux invariant. \end{thm}

Note that, on $\T^3 = \R^3/\Z^3$, cohomology classes can be represented by
constant differential forms and 1--dimensional homology classes by constant
vector fields. The slant product of the above theorem is then identified with
the interior product of vector fields with 2--forms.

The statement about torsion free contact structures is based on the
interaction between the action of the mapping class group and first homology
group of $\T^3$ on its \OS homology and ordinary cohomology. It sheds some light
on the sign ambiguity of the contact invariant since the sign ambiguity of the
Giroux invariant is very easy to understand.

\begin{cor}
There are infinitely many isomorphic contact structures whose isotopy
classes are pairwise distinguished by the \OS invariant. 
\end{cor}

Theorem \ref{thm:t3} proves, via gluing, a conjecture of Honda, Kazez and Matić
about the sutured invariants of $\S^1$--invariants contact structures on toric
annuli. This conjecture is stated in \cite{HKM_tqft}[top of page 35] and will be
discussed in Section \ref{S:t3} Proposition \ref{prop:annulus}.

Theorem \ref{thm:t3} also have some consequence for the hierarchy of
coefficients because $\Z_2$ coefficients can distinguish only finitely many
isotopy classes of contact structures (since $\HFhat(Y; \Z_2)$ is always
finite).

\begin{cor}
There exists a manifold on which the \OS invariant \over integer coefficients
distinguishes infinitely many more isotopy classes of contact structures than the
invariant \over $\Z_2$ coefficients.
\end{cor}

In the same spirit, we prove that twisted coefficients are more powerful than
$\Z$ coefficients even when the latter give non vanishing invariants.

\begin{prop}[see Propositions \ref{prop:annulus}]
There exist a sutured manifold with two contact structures having the same
non vanishing \OS invariant \over $\Z$ coefficients but which are distinguished
by their invariants \over twisted coefficients.
\end{prop}

In Section 1 we review the work of Giroux on certain contact structures on
circle bundles, the easy extension of this work to Seifert manifolds and torsion
calculations. In Section 2 we review \OS contact invariants. In Section 3 we
prove Theorem \ref{thm:t3}. In Section 4 we review the work of Honda, Kazez and
Matić on their contact TQFT and upgrade their $\SFH$ groups calculations to
twisted coefficients. In Section 5, by far the longest, we prove
\cite{HKM_tqft}[Conjecture 7.13] and the main theorem above.

\section{Partitioned contact structures on Seifert manifolds}
\label{S:partitions}

This section contains preliminary results in contact topology. We first recall
the crucial definition of Giroux torsion. The $k\pi$-torsion of a contact
manifold $(V, \xi)$ was defined in 
\cite{Giroux_2000}[Definition 1.2] to be the supremum of all integers $n \ge 1$
such that there exist a contact embedding of  
\[
\left(T^2 \times [0, 1], \ker\left(\cos(n k\pi z)dx - 
\sin(n k\pi z)dy\right)\right), \qquad (x,y,z) \in T^2 \times [0,1]
\] 
into the interior of $(V, \xi)$ or zero if no such integer $n$ exists.
Of course all $k\pi$--torsions can be recovered from the $\pi$--torsion. However
when we don't specify $k$ we mean $2\pi$-torsion. This is due to the fact that
only $2\pi$-torsion is known to interact with symplectic fillings and \OS
theory.

A multi-curve in an orbifold surface $B$ is a 1--dimensional submanifold
properly embedded in the regular part of $B$. When $B$ is closed, we will say
that a multicurve is essential in $B$ if none of its components bound a disk
containing at most one exceptional point.

Since we want to extend results from circle bundles to Seifert manifolds and
most surface orbifolds are covered (in the orbifold sense) by smooth surfaces,
the following characterization will be useful.

\begin{lem}
\label{lem:essential_curves}
Let $\Gamma$ be a multicurve in a closed orbifold surface $B$ whose (orbifold)
universal cover is smooth. The following statements are equivalent:
\begin{enumerate}
\item 
$\Gamma$ is essential;

\item
$\Gamma$ lifts to an essential multicurve in all smooth finite covers of $B$.

\item
$\Gamma$ lifts to an essential multicurve in some smooth finite cover of $B$.
\end{enumerate}
\end{lem}

\begin{proof}
We first prove (the contrapositive of) (1) $\implies$ (2).
Let $\Gamma$ be a essential multicurve in $B$ and $\pi$ be an orbifold covering
map from a smooth surface $\hat B$ to $B$. Suppose that a component of the inverse
image of $\Gamma$ bounds an embedded disk $\hat D$ in $\hat B$. Its image in $B$
is a topological disk $D$ and we only need to prove that this disk contains at
most one exceptional point. Using multiplicativity of the orbifold Euler
characteristic under the orbifold covering map from $\hat D$ to $D$, we get
$\chi(D) > 0$. This proves that $D$ contains at most one exceptional points
because its Euler characteristic is $1 - s + \sum_{i=1}^s 1/\alpha_i$ with
$\alpha_i \geq 2$ if it has $s$ exceptional points so $\chi(D) \leq 1 - s/2$. So
(1) implies (2). Since (2) obviously imply (3), we are left with proving
(the contrapositive of) (3) implies (1).

Assume that $\Gamma$ is not essential and let $D$ be a connected component of
the complement of $\Gamma$ in $B$ which is a disk with at most one exceptional
point. In any finite cover $\hat B$ of $B$, this disk lifts to a collection of
disks bounded by components of the lift of $\Gamma$ and containing at most one
exceptional point. So $\Gamma$ is non essential in all finite covers of $B$. 
\end{proof}

The following is the essential definition of this section.

\begin{defn}[obvious extension of \cite{Giroux_2001}]
  \label{def:partition}
A contact structure is partitioned by a multi-curve $\Gamma$ in $B$ if it
transverse to the fibers over $B \setminus \Gamma$ and if the surface
$\pi^{-1}(\Gamma)$ is transverse to $\xi$ and its characteristics are fibers.
\end{defn}

\begin{ex}[\cite{Lutz,Tsuboi_91}, see also \cite{Klaus_these}]
Let $V \to B$ be a Seifert manifold and $\Gamma$ be a non empty multi-curve in
$B$ whose class in $H_2(B, \partial B ; \Z_2)$ is trivial. There is a
$S^1$--invariant contact structure on $V$ which is partitioned by $\Gamma$.
This contact structure is unique up to isotopy among $S^1$--invariant contact
structures.
\end{ex}

The following theorem relies on \cite{gcs}[Theorem~A] and on easy extensions or
consequences of the fourth part of \cite{Giroux_2001}. Of course it also uses a
lot the results of \cite{Giroux_2000}. The two papers by Giroux can also be
replaced by the Honda versions \cite{Honda_I, Honda_II}. This theorem could be
easily improved to say things about Seifert manifolds with non empty boundary
but we won't need such improvements. Recall that a closed Seifert manifold is
small if it has at most three exceptional fibers and its base has genus zero.
Otherwise it is called large. In particular the bases of large Seifert manifolds
admit essential multi-curves. We denote by $e(V)$ the rational Euler number of a
Seifert manifold $V$. See \cite{gcs} for the conventions used here for Seifert
invariants and Euler numbers. In the statement we exclude for convenience the
(finitely many) Seifert manifolds which are torus bundles over the circles (see
for instance \cite{Hatcher_3D} to get the list).

\begin{thm}
  \label{thm:classif_partition}
Let $V$ be a closed oriented Seifert manifold over a closed oriented orbifold
surface.
\begin{enumerate}
\item 
A contact structure on $V$ partitioned by a multi-curve $\Gamma$ is
universally tight if and only if one of the following holds:
  \begin{enumerate}
  \item $\Gamma$ is empty
  \item $V$ is large and $\Gamma$ is essential
  \item $V$ is a Lens space (including $\S^3$ and $\S^2 \times \S^1$), 
	$e(V) \geq 0$, $\Gamma$ is connected and each component of its complement
	contains at most one exceptional point.
  \end{enumerate}

\item 
Any universally tight contact structure on $V$ is isotopic to a partitioned
contact structure.

\item 
Suppose $V$ is not a torus bundle over the circle. Let $\xi$ be a contact
structure on $V$ partitioned by an essential multi-curve $\Gamma$.  Let $n$ be
the greatest integer such that there exist $n$ closed components of $\Gamma$ in
the same isotopy class of curves. The Giroux torsion of $\xi$ is zero if
$\Gamma$ is empty and at most $\lfloor \frac{n}{2} \rfloor$ otherwise.

\item 
Let $\xi_0$ and $\xi_1$ be contact structures on $V$ partitioned by non empty
multi-curves denoted by $\Gamma_0$ and $\Gamma_1$ respectively.  If $\Gamma_0$
and $\Gamma_1$ are isotopic then $\xi_0$ and $\xi_1$ are so. If $\xi_0$ and
$\xi_1$ are isotopic and universally tight then $\Gamma_0$ and $\Gamma_1$ are
isotopic.
\end{enumerate}
\end{thm}

We first comment on some consequences of this theorem which have not much to do
with the main stream of the present paper.
We can deduce from it and \cite{Lisca_Matic} (or \cite{gcs}) the list (given in
corollary \ref{cor:ut_list} below) of Seifert manifolds which carry universally
tight contact structures. This list did not appear in the literature while the
(much subtler) list of Seifert manifolds which carry tight contact structures
(maybe virtually overtwisted) was obtained (with much more work) by P~Lisca and
A~Stipsicz in \cite{LS_existence}. In addition, the road taken in that paper to
prove existence on large Seifert manifold is much heavier than using the above
theorem (but the point of that paper is small manifolds).

\begin{cor}
\label{cor:ut_list}
A closed Seifert manifold $V$ admits a universally tight contact structure if and
only if one of the following holds:
\begin{enumerate}
	\item $V$ is large
	\item $V$ is a Lens space (including $\S^3$ and $\S^2 \times \S^1$)
	\item $V$ has three exceptional fibers which can be numbered
	such that its Seifert
	invariants are $(0, -2, (\alpha_1,\beta_1), (\alpha_2,\beta_2),
	(\alpha_3,\beta_3))$ with
	\[ \frac{\beta_1}{\alpha_1} > \frac{m-a}{m}, \quad 
 \frac{\beta_2}{\alpha_2} > \frac{a}{m}, \quad\text{and} \quad
 \frac{\beta_3}{\alpha_3} > \frac{m-1}{m}\]
 for some relatively prime integers $0 < a < m$.
\end{enumerate}
\end{cor}

The above theorem also proves that all universally tight contact structures on
Seifert manifolds interact nicely with the Seifert structure.

\begin{cor}
\label{cor:ut_char}
If $\xi$ is a universally tight contact structure on a closed Seifert manifold
$V$ then there exist a locally free $\S^1$ action on $V$ such that $\xi$
is either transverse to the orbits or invariant.
\end{cor}

Note that the alternative in the above corollary is not exclusive. A contact
structure which is both invariant and transverse to the orbits of a locally free
$\S^1$ action exists exactly when $e(V) < 0$, this was proved by Y~Kamishima and
T~Tsuboi in \cite{Tsuboi_91}. There is only one isomorphism class of contact
structure of this type when they exist. This class is of Sasaki type and
sometimes called the canonical isomorphism class of contact structures on $V$.

\begin{proof}[Proof of Theorem \ref{thm:classif_partition}]
We now outline the main differences between Theorem \ref{thm:classif_partition}
and the parts which are already written in \cite{Giroux_2001}. First it should
be noted that, when $V$ is either a Lens space or a solid torus with a standard
Seifert fibration, everything is well understood thanks to the classification
theorems of \cite{Giroux_2000} (see also \cite{Honda_I}). So we don't consider
these Seifert manifolds in the following.

1) 
Let $\xi$ be a contact structure on a closed $V$ partitioned by $\Gamma$.
If $\Gamma$ is empty then $\xi$ is transverse to the fibers hence
universally tight according to \cite{gcs}[Theorem A] (this direction follows
rather directly from Bennequin's theorem). If $V$ is large and $\Gamma$ is
essential then the base $B$ of $V$ is covered (in the orbifold sense) by a
smooth surface $\Sigma$ and there is a corresponding circle bundle 
$\hat V \to \Sigma$ covering (honestly) $V$. The pulled back contact structure
is partitioned by the inverse image of $\Gamma$ which is essential according to
Lemma \ref{lem:essential_curves} so $\xi$ is universally tight according to
\cite{Giroux_2001} (first line of page 252).

Conversely, assume that $\xi$ is universally tight and partitioned by a non
empty multi-curve $\Gamma$. Assume first the base of $V$ is covered by a smooth
surface of genus at least one (for instance if $V$ is large). The manifold $V$
then is covered by a circle bundle over that surface as above. We get from
\cite{Giroux_2001}[Theorem 4.4] that the lifted contact structure is partitioned
by a multi-curve, unique up to isotopy, which is essential.  Since the lift of
$\Gamma$ is such a curve, it is essential and Lemma \ref{lem:essential_curves}
implies that $\Gamma$ is also essential. In particular $V$ is large.

If no such cover of the base exists (and $V$ is not a Lens space) then its base
$B$ is a sphere with exceptional points of order $(2, 2, n)$, $(2, 3, 3)$, $(2, 3, 4)$
or $(2, 3, 5)$ (see \cite{Thurston_chap_13}[Theorem 13.3.6]). In each case 
$B$ is covered by $\S^2$ and all
curves in the regular locus of $B$ bounds a disk whose pre-image in $\S^2$ is
disconnected so $\xi$ is virtually overtwisted according to
\cite{Giroux_2001}[Proposition 4.1 and Lemma 4.7].

2) 
Recall that a contact structure on a Seifert manifold is said to have
non-negative maximal twisting number\footnote{some texts say zero twisting
number in this case} if it is isotopic to a contact structure for which there
exists a Legendrian regular fiber whose contact framing coincides with the
fibration framing. If this property is not satisfied then \cite{gcs}[Theorem A]
ensures that any universally tight $\xi$ is isotopic to a contact structure
partitioned by the empty multi-curve (i.e.  transverse to the fibers). We now
assume that $\xi$ has non negative maximal twisting number and has been isotoped
so that it admits a Legendrian fiber $L$ as above. Let $K$ be a wedge of circles
based at $L$ in the smooth part of $B$ (seen as the space of all fibers) let $R$
be a small regular neighborhood of $K$. We can choose $K$ an $R$ such that the
complement $R'$ of $R$ in $B$ is made of disks containing exactly one
exceptional point. The techniques of \cite{Giroux_2001} prove that $\xi$ is
isotopic to a contact structure which, over $R$ is partitioned by a multicurve
$\Gamma_R$ which intersects all boundary components of $R$. We now assume this
property. Let $V'$ denote the (non necessarily connected) Seifert manifold over
$R'$ and $\xi'$ the restricted contact structure. Since $\Gamma_R$ intersects
all components of $\partial R$, each component $T$ of the boundary of $V'$
contains a Legendrian regular fiber which is either a closed leaf or a circle of
singularities of the characteristic foliation $\xi' T$. If $\xi'$ is universally
tight then the classification of universally tight contact structures on solid
tori directly implies that $\xi'$ is $\partial$--isotopic to a contact structure
partitioned by some $\Gamma_{R'}$ extending $\Gamma_R$ and we are done. More
precisely, for each component $W$ of $V'$, this classification guaranties the
existence of exactly one isotopy class of universally tight contact structure
coinciding with $\xi'$ on $\partial W'$ when $W$ contains no exceptional fiber
and two otherwise. In the latter case, the two classes correspond to the two
isotopy classes of arcs extending $\Gamma_R$ inside the base of $W$ (which is a
disk with one exceptional point).

So it remains to prove that if $\xi$ has non negative maximal twisting number
and is universally tight then each solid torus $W$ isotopic to a fibered one has
a universally tight induced contact structure. This is obvious if the universal
cover $\tilde W$ of $W$ naturally embeds into the universal cover $\tilde V$ of
$V$.  This $\tilde V$ can be built in two stages: first one takes the (orbifold)
universal cover of the base $B$ and pulls back the Seifert fibration and then
one unwraps the fibers as much as possible. The sought embedding of $\tilde W$
obviously exist when the fibers can be completely unwrapped. Due to the
classification of orbifolds surfaces the only problematic case if one excludes
Lens spaces is when $\tilde V$ is $\S^3$ with its (smooth) Hopf fibration. But,
by definition of tightness, any tight contact structure on $\S^3$ has 
negative twisting number with respect to the Hopf fibration so this case does
not happen here (the property of having non negative twisting number is
obviously inherited by finite covers using lifts of isotopies).

3) 
Since we assume that $V$ is not a torus bundle over the circle, all
incompressible tori are isotopic to fibered ones (see e.g. \cite{Hatcher_3D}).

Suppose first that $\xi$ is partitioned by the empty multicurve (i.e. is
transverse to all fibers).  It was proved in \cite{gcs}[Theorem A] that such a
contact structure has negative maximal twisting number. Suppose by contradiction
that it has non vanishing $\pi$--torsion. Up to isotopy of $\xi$ there is an
annulus in the base which is foliated by circles $(C_t)_{t \in [0,1]}$ such
that,
\begin{itemize}
\item 
For all $t$, the torus $T_t$ above $C_t$ in $V$ is prelagrangian.
\item
The directions of the Legendrian foliations of the $T_t$ go all over the
projective line .
\end{itemize}
During this full turn around the projective line, the Legendrian direction meets
the fiber direction and there are Legendrian curve whose contact framing
coincides with the fibration framing so we get a contradiction with the maximal
twisting number estimate. 

We now assume that $\xi$ is partitioned by a non empty multicurve $\Gamma$ and
that no two components of $\Gamma$ are isotopic.  Incompressible fibered tori
correspond to essential curves in the base orbifold $B$. To any such curve $C$
correspond an orbifold covering of $B$ by an open annulus $\hat B$ and the
Seifert fibration lifts to a trivial (smooth) circle fibration $\hat V$. The
lifted contact structure is partitioned by the inverse image of $\Gamma$ which
is made of as many essential circles as there were components of $\Gamma$
isotopic to $C$ (at most $n$) and lines properly embedded in $\hat B$. If there
exist a contact embedding of a toric annulus with its standard torsion 
contact structure in $V$ then it lifts to $\hat V$ inside some $K \times \S^1$
with $K \subset \hat B$ compact. The classification of tight contact structures
on toric annuli forbids torsion higher than $\lfloor \frac{n}{2} \rfloor$
knowing the partition we have over $K$. This argument is not new, it was
explained to me (around 2005) by E~Giroux.

4) The first part is a straightforward extension of \cite{Giroux_2001}[Lemma
4.7]. Suppose now that $\xi_0$ and $\xi_1$ are isotopic. If $V$ is not large
then we are in case (c) of the first point so that $\Gamma_0$ and $\Gamma_1$ are
trivially isotopic. So we now assume that $V$ is large. In particular $\Gamma_0$
and $\Gamma_1$ are essential. By definition, they are isotopic in $B$ if and
only if they are isotopic in the smooth surface $R$ obtained from $B$ by
removing a small open disk around each exceptional point. By definition of
essential curves, no component of $\Gamma_0$ or $\Gamma_1$ is parallel to the
boundary of $R$.  According to W~Thurston, $\Gamma_0$ and $\Gamma_1$ are
isotopic if and only if they have the same geometric intersection number with
all closed curves in $B$ \cite{Thurston_surfaces}[Proposition page 421]. These
geometric intersections number have a contact topology interpretation explained
in \cite{Giroux_2001}[Section 4.E] which proves they are invariant under contact
structures isotopy exactly as in the circle bundle case.
\end{proof}

\section{Contact invariants in sutured Floer homology}

In this section we review sutured Heegaard--Floer homology and
the contact invariants which lives in it.

Heegaard--Floer homology was introduced by P~Ozsváth and Z~Szabó in
\cite{OSz_first} and extended to sutured manifold by A~Juhász in
\cite{Juhasz_def}. In the following we will often silently identify a closed
manifold $M$ with the sutured manifold $(M \setminus B^3, S^1)$ and use sutured
Floer theory (SFH) also in this case.

We denote the universal twisted $\twSFH(-M,\Gamma; \Z[H_2(M ; \Z)])$ by
$\twSFH(-M,\Gamma)$ and, whenever there is no ambiguity on the manifold $M$ we
are considering, we denote $\Z[H_2(M ; \Z)]$ by $\LL$.

According to \cite{GH}[Lemma 10], if a contact invariant vanishes in $\twSFH$
then it vanishes for all coefficients rings.

\begin{thm}[Ozsváth--Szabó, Honda--Kazez--Matić, Ghiggini--Honda--Van Horn Morris]
  \label{thm:ci}
Let $(M, \Gamma)$ be a balanced sutured manifold. To each contact structure
$\xi$ on $(M, \Gamma)$, one can associate a contact invariant which is a set
$c(\xi)$ in $\SFH(-M, -\Gamma)/\pm 1$ and a twisted contact invariant which is a
set $\twc(\xi)$ in $\twSFH(-M, -\Gamma)/\Lx$ satisfying the following
properties:
\begin{enumerate}
  \item	the set $\twc(\xi)$ is invariant under $\partial$--isotopy of $\xi$
  \item if $\xi$ is overtwisted then $\twc(\xi) = 0$
  \item if $\xi$ has non zero torsion then $c(\xi) = 0$
  \item if $M$ is closed and $\xi$ is weakly fillable then $\twc(\xi) \neq 0$
  \item if $M$ is closed and $\xi$ is strongy fillable then $c(\xi) \neq 0$
  \item if $(M',\Gamma')$ is a sutured submanifold of $(M,\Gamma)$ and $\xi$ is a contact
	structure on $(M \setminus M',\Gamma \cup \Gamma')$ then there exists a linear map 
	\[
	\twPhi_{\xi} : \twSFH(-M', -\Gamma') \to \twSFH(-M,-\Gamma)
	\]
	such that, for any contact structure $\xi'$ on $(M', \Gamma')$, one has
	\[
	\twc(\xi \cup \xi') = \Phi_{\xi}(\twc(\xi')).
	\]
	If every connected component of $M \setminus int(M')$ intersect $\partial M$
	then there are analogous maps \over $\Z$ coefficients. They are denoted
	without underlines.
  \item if $(M',\xi')$ is a contact submanifold of $(M, \xi'\cup\xi)$ then 
	$\twc(\xi') = 0$ implies $\twc(\xi \cup \xi') = 0$ and analogously \over $\Z$
	coefficients.
\end{enumerate}
\end{thm}

The construction of the contact invariants (and the isotopy invariance) can be
found in \cite{OSz_contact} for the closed case and \cite{HKM_sci} in general.
The fact that it vanishes for overtwisted contact structures was first proved
for the closed case and untwisted coefficients in \cite{OSz_contact} and follow
in general from the last property and the explicit calculation of the twisted
contact invariant of a neighborhood of an overtwisted disk found in
\cite{HKM_sci}. The assertion about torsion was proved in \cite{GHV}. Both
assertions about fillings are consequences of \cite{OSz_genus_bounds}[Theorem
4.2], using the fact that, for strong fillings, the coefficient ring in this
theorem reduces to $\Z$ (see also \cite{Ghiggini_strong_weak}[Theorem 2.13] for
an alternative proof of the strong filling property). The gluing properties are
proved in \cite{HKM_tqft} for untwisted coefficients and extended to twisted
coefficients in \cite{GH}. The gluing maps are unique up to multiplication by an
invertible element of the relevant coefficients ring. Such maps will be called
HKM gluing maps. 

There is one piece of structure of Heegaard--Floer theory which doesn't seem to
have been explicitly discussed\footnote{we don't claim to do anything new in
this paragraph, but we can't find a reference for it} in our context up to now:
the mapping class group action. Any diffeomorphism of a 3--manifold $M$ acts on
any variant of $HF(M)$.  Here we need to be precise about what depends on the
way a Heegaard diagram is embedded inside a manifold and what does not depend on
it.  The usual way to do that is to consider embedded Heegaard diagrams as pairs
made of a self-indexing Morse function with unique minima and maxima and one of
its Morse--Smale pseudo-gradients. Given such a pair $(f, X)$, the Heegaard
surface is $f^{-1}(3/2)$ and the Heegaard circles are the intersections of the
stable or unstable disks of the index 1 and 2 critical points. We denote the
group associated to $(f, X)$ by $HF(f,X)$ (we can use here $\HFhat$,
$\HFp$,\dots).  Let $\varphi$ be a diffeomorphism of $M$. Then
\cite{OSz_triangles}[Theorem 2.1] gives an isomorphism 
\[
\Psi : HF(f,X) \to HF(f \circ \varphi, \varphi^* X)
\]
which is well defined up to sign. But of course the diffeomorphism $\varphi$
also gives an isomorphism between the corresponding abstract Heegaard diagrams
which then gives an isomorphism $\Phi$ between Heegaard--Floer groups. The
action of $\varphi$ on $HF(f, X)$ is defined to be $\Phi^{-1} \circ \Psi$.
It is obvious from the construction that the contact invariant is equivariant
under this action. What is not obvious is that isotopic diffeomorphisms have the
same action so that we get an action of the mapping class group. This has been
checked by P~Ozsváth and A~Stipsicz in the context of knot Floer homology in
\cite{OS_knots}. In this paper we don't use this invariance but use
specific diffeomorphisms. Actually this invariance should never be needed in
contact geometry since we already know that the contact invariant is a contact
structure isotopy invariant so that diffeomorphism isotopy invariance is
automatic on the subgroup spanned by contact invariants in any $\HFhat$ or
$\HFp$.

\section{Contact structures on the three torus}
\label{S:t3}

In this section we prove Theorem \ref{thm:t3} from the introduction.
The following easy lemma is the key algebraic trick.

\begin{lem} 
If an isomorphism $\Phi : \HFhat(\T^3) \to H^1(\T^3) \oplus H^2(\T^3)$ is
$H_1(\T^3)$--equivariant then it conjugates the $\SL_3$ actions of both sides.  
\end{lem}

\begin{proof}
In this proof we drop $\T^3$ from the notations. We denote by $\rho$ the
canonical action of $\SL_3$ on $H_1$.
Let $\rho_1$ and $\rho_2$ be two representations of $SL_3$ on $H^1 \oplus H^2$
which are compatible with the $H_1$ action, that is:
\[
\forall g \in \SL_3, \gamma \in H_1, m \in H^1 \oplus H^2,\quad 
\left(\rho(g)\gamma\right) \rho_i(g)m = \rho_i(g)\left(\gamma m\right).
\]
We want to prove that $\rho_1 = \rho_2$ since this, applied to the standard
action and to the action transported by $\Phi$, will prove the proposition.

We first prove that, for all $g \in \SL_3$, $\rho_1(g)$ and $\rho_2(g)$ agree on
$H^2$. The key property of the $H_1$ action is that it separates all elements of
$H^2$: for all $m \neq m' \in H^2$, there exists $\gamma$ in $H_1$ such that 
$\gamma m = 0$ and $\gamma m' \neq 0$.

Suppose by contradiction that there exists $g \in \SL_3$ and $m \in H^2$ such
that $\rho_1(g)m \neq \rho_2(g)m$. According to the separation property, there
exists $\gamma'$ in $H_1$ such that $\gamma'\rho_1(g)m = 0$ and
$\gamma'\rho_2(g)m \neq 0$. Setting $\gamma = \rho(g)^{-1}(\gamma')$, we get
$\rho(g)\gamma\rho_1(g)m = 0$ and
$\rho(g)\gamma\rho_2(g)m \neq 0$, so $\rho_1(g)(\gamma m) = 0$ and
$\rho_2(g)(\gamma m) \neq 0$, which is absurd since $\rho_1(g)$ and $\rho_2(g)$
are both isomorphisms.

We now prove that the representations agree on $H^1$. For all $m' \in H^1$,
there exists $m \in H^2$ and $\gamma \in H_1$ such that $m'  = \gamma m$. So for
any $g \in SL_3$ and $i = 1,2$, we get 
$\rho_i(g) m' = \rho_i(g) (\gamma m) = \rho(g)\gamma \rho_i(g)m$ and we know
that $\rho_1(g)m = \rho_2(g)m$ thanks to the first part so 
$\rho_1(g)m' = \rho_2(g)m'$.
\end{proof}

\begin{proof}[Proof of Theorem \ref{thm:t3}]
The existence of such an isomorphism is Proposition~8.4 of \cite{OSz_absolute}.
The above lemma proves that, for any $\Phi$ as in the statement and any $x \in
\HFhat$, $x$ and $\Phi(x)$ have the same stabilizer under the action of $\SL_3$.
The uniqueness of $\Phi$ follows since primitive elements of $H^1 \oplus H^2$
are characterized up to sign by their stabilizers. Linearity of $\Phi$ then
guaranties that the sign is common to all elements.

We now prove that the Poincaré dual of the Giroux invariant and the image of
the \OS invariant coincide on torsion free contact structures. First remark
that the \OS invariant belongs to $\HFhat_{-1/2} \simeq H^1$ because the Hopf
invariant of tight contact structures on $\T^3$ is $1/2$.
So both invariants are primitive elements of $H^1$. We prove that the stabilizer
of $G(\xi)$ is contained in that of $c(\xi)$ using equivariance of both
invariants and the fact that $G$ is a total invariant. For any $g$ in $\SL_3$
and $\xi$ a torsion free contact structure, we have
\[
\begin{split}
g G(\xi) = G(\xi) & \iff G(g\xi) = G(\xi) \\
&\iff g\xi \sim \xi\\
&\implies c(g\xi) = c(\xi)\\
&\iff g c(\xi) = c(\xi)
\end{split}
\]
so we have the annouced inclusion of stabilizers and this gives 
$c(\xi) = G(\xi)$.
\end{proof}

\section{The contact TQFT}

We now review the contact TQFT of Honda--Kazez--Matić.
Let $\Sigma$ be a non necessarily connected compact oriented surface with non
empty boundary. Let $F$ be a finite subset of $\partial \Sigma$ whose
intersection with each component of $\partial \Sigma$ is non empty and consists
of an even number of points. We assume that the components of 
$\partial \Sigma \setminus F$ are labelled alternatively by $+$ and $-$. This
labelling will always be implicit in the notation $(\Sigma, F)$. 
The contact TQFT associates to each $(\Sigma, F)$ the graded group
\[
V(\Sigma, F) = SFH(-(\Sigma \times \S^1), -(F \times \S^1))\]
(strictly speaking, one should replace $F$ by a small translate of $F$ along
$\partial \Sigma$ in this formula).

In this construction one can use coefficients in $\Z_2$ or twisted coefficients
(including the trivial twisting which leads to $\Z$ coefficients). We denote by
$\twV(\Sigma, F)$ the version twisted by $\Z[H_2(\Sigma \times \S^1)]$.

\begin{prop}
  \label{prop:twV}
Let $(\Sigma, F)$ be a surface with marked boundary points as above and $\M$ be
any coefficient module for the sutured manifold 
$(\Sigma \times \S^1, F \times \S^1)$. We have, for any coherent orientations
system:
\[
\twV(\Sigma, F; \M)
\simeq (\M_{(-1)} \oplus \M_{(1)})^{\otimes(\# F/2 - \chi(\Sigma))}.
\]
The subscripts $(-1)$ and $(1)$ refer to the grading.
\end{prop}

\begin{proof}
The analogous statement \over $\Z$ coefficients was proved in \cite{HKM_tqft}
using product annuli decomposition, \cite{Juhasz_decat}[Proposition 7.13]. This
technology is not yet available \over twisted coefficients but one can actually
draw explicit admissible sutured Heegaard diagrams with vanishing differential
for these sutured manifolds. We will sketch how to construct them and draw
pictures for the three cases where we actually use this computation below.

We first recall what is an (embedded) Heegaard diagram for a (balanced
connected) sutured manifold $(V, \Gamma)$. It consists of a surface $S$ properly
embedded in $V$ and circles $\alpha_1, \beta_1,\dots, \alpha_k, \beta_k$ in $S$
such that:
\begin{itemize}
\item  
$\partial V = \Gamma$

\item 
if we denote by $V_+$ the connected component of $V \setminus S$ containing
$R_+$, there exist open disks properly embedded in $V_+$, called compression
disks, bounded by the $\alpha$ circles and such that $V_+ \cup R_+$ retracts by
deformation on $R_+$

\item 
the analogous statement holds for $V_-$ and $R_-$ with the $\beta$ circles.
\end{itemize}

We now return to the proposition.
Let $g$ be the genus of $\Sigma$, $r$ the number of boundary components and $n =
\# F/2$. The sutured manifold we study will be denoted by $(V, \Gamma)$ for
concision.  We rule out the trivial $(g = 0, r = 1, n = 1)$ case from this
discussion as it needs (easy) special treatment. 
Assume first that $r = 1$ and $n = 1$. Let $a_1,\dots, a_{2g}$ be a
system of disjoints arcs properly embedded in $\Sigma$ which cuts $\Sigma$ to a
disk.  Let $P_1, \dots, P_{2g}$ be tubes around the arcs $a_i \times \{ \theta_0
\}$ for some fixed $\theta_0 \in \S^1$. We can assume that each $P_i$ meets the
boundary of $V$ in its positive part $R_+$. Let $S'$ be the union of $R_+$ and
the tubes $P_i$. The surface $S$ obtained by pushing $S'$ to make it properly
embedded in $V$ is a Heegaard surface for $(V, \Gamma)$. 

Each tube $P_i$ naturally bounds a regular neighborhood $D \times [-1, 1]$ of the arc
$a_i$. Let $\alpha_i$ be the boundary of $D \times \{0\}$ in each $P_i$. Let
$\beta_i$ be the union of $\{ \pm 1\} \times [-1, 1]$ and two arcs in $R_+$ so
that $\beta_i$ and half of $P_i$ becomes isotopic to a fibered annulus in $V$.
See figure \ref{fig:hs_punct_torus} for the case $(g = 1, r = 1, n = 1)$.
\begin{figure}[ht]
	\begin{center}
		\includegraphics[width=8cm]{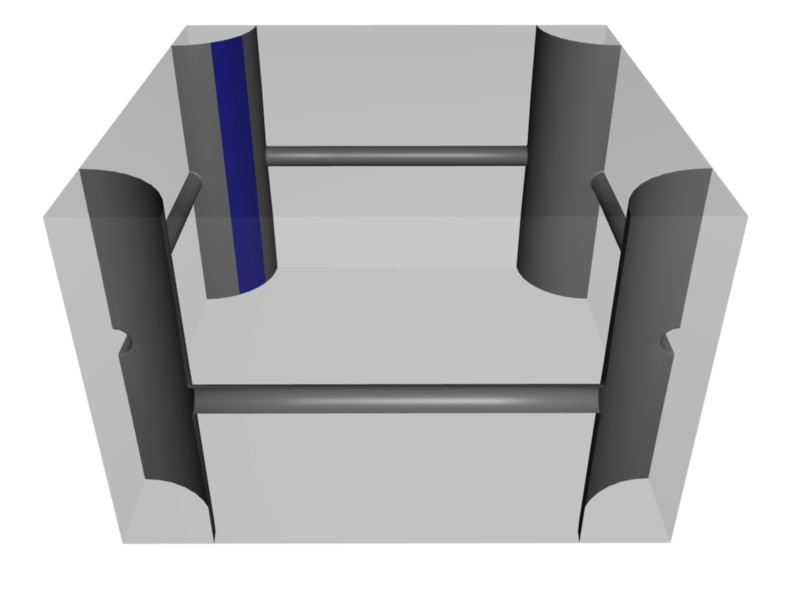}
	\end{center}
	\caption{Heegaard surface for $V = (\T^2 \setminus D^2) \times \S^1$ with two
	vertical sutures. Everything lives inside a cube minus a neighborhood of its
	vertical edges. The cube's faces are pairwise glued to get $(\T^2 \setminus
	D^2) \times \S^1$. The picture show the (almost) Heegaard surface $S'$ of the
	proof. The blue annulus in the back left is the negative part $R_-$ of the
	boundary of $V$. The $\beta$ compression disks are inside the front, back,
	left and right faces of the cube (which get glued to two annuli in $V$).}
	\label{fig:hs_punct_torus}
\end{figure}
We then have a Heegaard diagram $(S, \alpha, \beta)$ for $(V, \Gamma)$. We now
explain what happens when we add some extra boundary components (i.e. $r > 1$).
For each extra component $T_j$ we add two tubes $P_{2g + j}$ and $P'_{2g +j}$
around horizontal arcs $a_{2g+j} \times \{\theta_0\}$ and $a'_{2g+j} \times
\{\theta_0\}$. We choose these arcs so that they can be completed by arcs in the
positive part of $\partial \Sigma$ to get a circle isotopic to the new boundary
component. See figure \ref{fig:hs_annulus} for the case $(g = 0, r = 2, n = 2)$
where the extra boundary component is the front one. 
\begin{figure}[ht]
	\begin{center}
		\includegraphics[width=8cm]{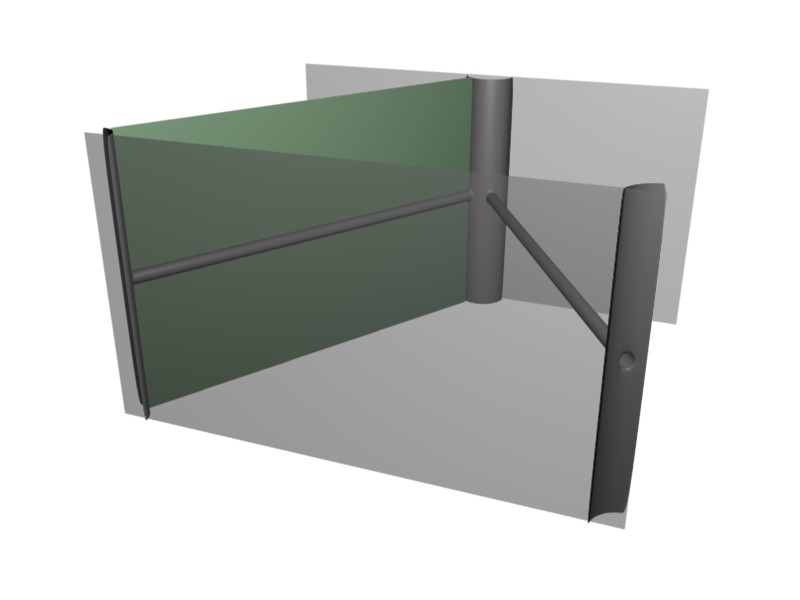}
	\end{center}
	\caption{Heegaard surface for $V = (I\times\S^1) \times \S^1$ with two
	vertical sutures. Top and bottom are glued. Left and right are glued. The boundary of $V$ is the
	union of the transparent tori (drawn as rectangles). The Heegaard surface is the union of
	two vertical annuli and two horizontal tubes. A compression disk bounded by a
	$\beta$ curve is shown in green. }
	\label{fig:hs_annulus}
\end{figure}
We add circles $\alpha_{2g+j}$, $\alpha'_{2g+j}$, $\beta_{2g+j}$ and
$\beta'_{2g+j}$ to the diagram as above. When there are extra marked points on
the boundary (i.e. $n > r$), we add one tube $P_{2g + r - 1 + k}$ between two
positive parts of the relevant boundary component. We add the corresponding
circles to the diagram. See figure \ref{fig:hs_bypass} for the case 
$(g = 0, r = 1, n = 3)$ where the extra sutures are the front ones.
\begin{figure}[ht]
	\begin{center}
		\includegraphics[width=8cm]{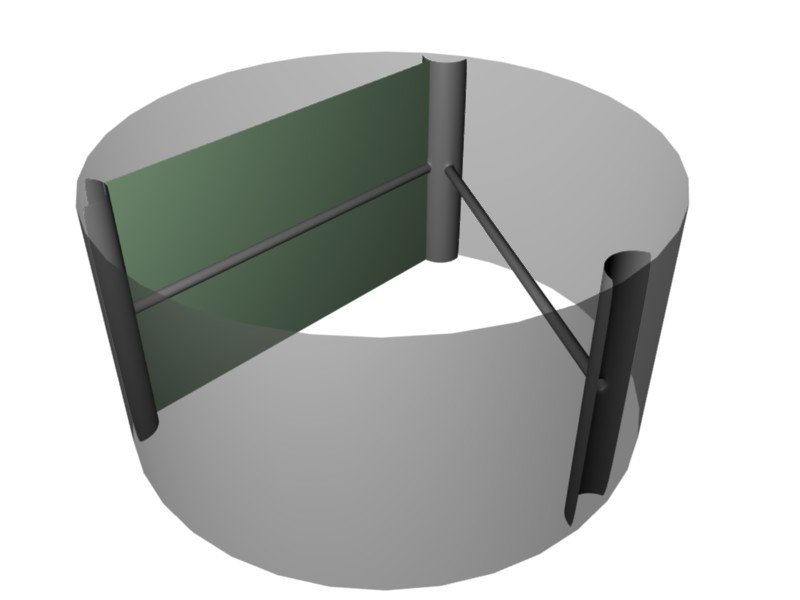}
	\end{center}
	\caption{Heegaard surface for $V = D^2 \times \S^1$ with three
	vertical sutures. Top and bottom are glued. The boundary of $V$ is the
	transparent torus (drawn as an annulus). The Heegaard surface is the union of
	two vertical annuli and two horizontal tubes. A compression disk bounded by a
	$\beta$ curve is shown in green. }
	\label{fig:hs_bypass}
\end{figure}
In this paragraph, whenever we started from the trivial case 
$(g = 0, r = 1, n = 1)$ which was ruled out above, we can use as a starting
point the degenerate diagram with Heegaard surface $R_+$ and no circle.

The constructed diagrams have $2g + 2(r-1) + (n - r)$ circles of each type and 
$\# \alpha_i \cap \beta_j = 2 \delta_{ij}$. Hence the chain complex has rank
$2^{n - \chi(\Sigma)}$. So the proposition follows from the admissibility of
these diagrams and the vanishing of the associated differentials.

Each arc $a_i$, $1 \leq i \leq 2g$ can be extend to a loop $\bar a_i$ and each
pair of arcs corresponding to extra boundary components can be extended to a
loop $l_j$, $1 \leq r - 1$ such that the collection of tori $\bar a_i \times
\S^1$ and $l_j \times \S^1$ gives a basis of $H_2(V, \Z)$.  This basis can be
realized by periodic domains using the $\alpha$ and $\beta$ circles associated
to the corresponding arcs.  So we have a basis of $H_2(V, \Z)$ associated to
disjoint periodic domains, each having both positive and negative coefficients.
Since they have disjoint support, any linear combination of these domains will
be admissible and the diagram is admissible. 

To compute the differential we note that each region of the complement of the
circles in $S$ which is not the base region is either a rectangle or
an annulus. In addition each rectangle is adjacent to either a rectangle using
the same circles or to the base region or to an annulus. One can then use
Lipshitz's formula to prove that the Heegaard--Floer differential vanishes.
\end{proof}

A dividing set for $(\Sigma, F)$ is a multi-curve $K$ in $\Sigma$ (see
Definition \ref{def:partition}). The complement of a dividing set in $\Sigma$
splits into two (non connected) surfaces $R_\pm$ according to the sign of their
intersection with $\partial \Sigma$. The graduation of a dividing set is defined
to be the difference of Euler characteristics $\chi(R_+) - \chi(R_-)$.

A dividing set $K$ is said to be isolating if there a
connected component of the complement of $K$ which does not intersect the
boundary of $\Sigma$.

To each dividing set $K$ for $(\Sigma, F)$ is associated the contact invariant
of the contact structures partitioned by $K$. All such contact structures are
either isotopic according to Theorem \ref{thm:classif_partition} or overtwisted
so they have the same invariant. These invariants belong to the graded part
given by the graduation of $K$.

\begin{thm}[\cite{HKM_tqft}]
Over $\Z_2$ coefficients, the following are equivalent:
\begin{enumerate}
  \item $c(K) \neq 0$
  \item $c(K)$ is primitive
  \item $K$ is non isolating
\end{enumerate}
Over $\Z$ coefficients, (3) $\implies$ (2) $\implies$ (1).
\end{thm}

Conjecture 7.13 of \cite{HKM_tqft} states that the assertions in this theorem
are equivalent \over $\Z$ coefficients. What remains to be proved is that
isolating dividing sets have vanishing invariant. This (and more) will be proved
in Section \ref{S:vanishing}.

\section{Vanishing results}

\label{S:vanishing}

In this section we prove the main theorem from the introduction and the
following theorem which finishes off the proof of Conjecture 7.13 of
\cite{HKM_tqft}. We use the definitions and notations of the previous section.

\begin{thm}
  \label{thm:tqft}
  If $K$ is isolating then $c(K)=0$ \over $\Z$--coefficients.
\end{thm}

Note that the analogous statement \over twisted coefficients is known to be
false. For instance if we consider on $\T^3$ a contact structure partitioned by
four essential circles and remove a small disk meeting one of these circles
along an arc then we get an isolating dividing set on a punctured torus whose
twisted invariant is sent to a non vanishing invariant according to Theorem
\ref{thm:ci} since the corresponding contact structures on $\T^3$ are weakly
fillable.

\begin{defn}
  \label{def:bypass}
We say that dividing sets $K_0$, $K_1$ and $K_2$ are bypass-related if they
coincide outside a disk $D$ where they consists of the dividing sets of
Figure \ref{fig:bypass}.
\end{defn}

\begin{figure}[ht]
  \begin{center}
  	\includegraphics{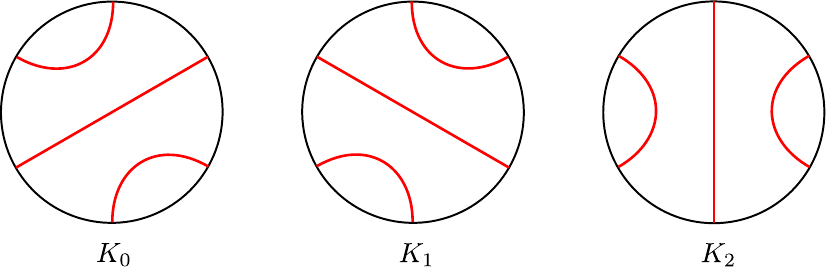}
  \end{center}
  \caption{Bypass relation}
  \label{fig:bypass}
\end{figure}

The following lemma is essentially proved in \cite{HKM_tqft} in the combination
of proofs of Lemma 7.4 and Theorem 7.6. We write a proof here to explain why
twisted coefficients come for free.

\begin{lem}
  \label{lem:bypass}
If $K_0$, $K_1$ and $K_2$ are bypass-related then, for any representatives
$\tilde \twc_i \in \twc(K_i)$, there exist 
$a, b \in \Lx$
such that $\tilde \twc_0 = a \tilde \twc_1 + b \tilde \twc_2$.
The same holds \over $\Z$ coefficients.
\end{lem}

\begin{proof}
The first part of the proof concentrate on the disk where the dividing sets
differ. Let $\tilde c_i^D$ be representatives of the contact invariants of the
three dividing sets on a disk $D$ involved in Definition \ref{def:bypass}. Note
that $H_2(D \times \S^1)$ is trivial so we now work \over $\Z$ coefficients and
suppress the underlines.

Because the $c_i^D$'s all belong to the same rank 2 summand of $V(D, F_D)$ there
are integers $\lambda$, $\mu$ and $\nu$ not all zero such that 
\begin{equation}
  \label{eqn:lmn}
\lambda \tilde c_0^D = \mu \tilde c_1^D + \nu \tilde c_2^D. 
\end{equation}

We denote by $K_\pm$ the dividing sets of Figure \ref{fig:K_plus_minus} and by
$c_\pm$ their contact invariants.
\begin{figure}[ht]
  \begin{center}
  	\includegraphics{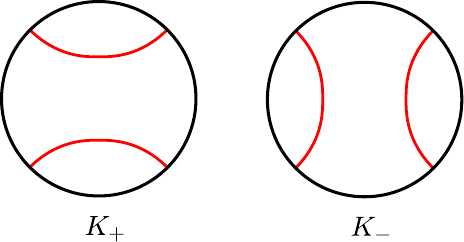}
  \end{center}
  \caption{Dividing sets used to prove Lemma \ref{lem:bypass}}
  \label{fig:K_plus_minus}
\end{figure}

Label the points of $F_D$ clockwise by $1,\dots, 6$ starting with the upper
right point. Let $\Phi_j$, $j = 1, 2, 3$, denote a HKM gluing map obtained by
attaching a boundary parallel arc between points $j$ and $j+1$.  The gluing maps have
the following effects:
\begin{align}
  \label{eqn:phi1}  
  \Phi_1 : c_0^D \mapsto c_+, \quad c_0^D \mapsto c_+, \quad c_0^D \mapsto 0 \\ 
  \label{eqn:phi2}  
  \Phi_2 : c_0^D \mapsto 0, \quad c_0^D \mapsto c_-, \quad c_0^D \mapsto c_- \\ 
  \label{eqn:phi3}  
  \Phi_3 : c_0^D \mapsto c_+, \quad c_0^D \mapsto 0, \quad c_0^D \mapsto c_+ 
\end{align}
Using these equations and the facts that $c_\pm$ are non zero in a torsion
free group (see Proposition \ref{prop:twV}), we get
\begin{gather*}
(\ref{eqn:phi1}) \implies \lambda = \pm \mu\\
(\ref{eqn:phi2}) \implies \mu = \pm \nu\\
(\ref{eqn:phi3}) \implies \lambda = \pm \nu
\end{gather*}
and they are all non zero so we can divide equation 
\ref{eqn:lmn} by $\lambda$ to get 
\begin{equation}
  \label{eqn:epsilons}
\tilde c_0^D = \varepsilon_1 \tilde c_1^D + \varepsilon_2 \tilde c_2^D. 
\end{equation}
with
$\varepsilon_1=\mu/\lambda$ and $\varepsilon_2=\nu/\lambda$. 

We now return to our full dividing sets.
Let $D$ be the disk  where the $K_i$'s differ. Denote by $F_D$ the (common)
intersection of the $K_i$'s with $\partial D$.
Let $\xi_0$, $\xi_1$ and $\xi_2$ be contact structures partitioned by $K_0$,
$K_1$ and $K_2$ respectively and coinciding with some $\xi_b$ outside
$D\times\S^1$.

Let $\twPhi : V(D, F_D) \to \twV(\Sigma, F)$ be a HKM gluing map associated to
$\xi_b$. According to Theorem \ref{thm:ci}, there exist invertible elements $a_i$ of $\LL$
such that $\twPhi(\tilde c_i^D) = a_i \tilde \twc_i$ for all $i$. 
We now apply $\twPhi$ to equation \ref{eqn:epsilons} 
and put $a=\varepsilon_1 a_1 a_0^{-1}$ and 
$b=\varepsilon_2 a_2 a_0^{-1}$
\end{proof}

Using this Lemma, we can reprove the main result of \cite{GHV}.

\begin{prop}[\cite{GHV}]
\label{prop:GHV}
Contact structures with positive Giroux torsion have vanishing contact invariant
over $\Z$ coefficients.
\end{prop}

\begin{figure}[ht]
  \begin{center}
  	\includegraphics{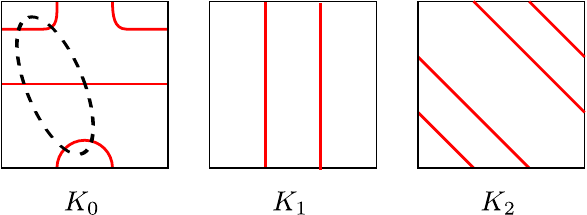}
  \end{center}
  \caption{Dividing sets for Propositions
  \ref{prop:GHV} and \ref{prop:annulus}. Left and right sides of each squares
  should be glued to get annuli.}
  \label{fig:div_GHV}
\end{figure}

\begin{proof}
Let $(A, F_A)$ be an annulus with two marked points on each boundary component
and consider the dividing sets of Figure \ref{fig:div_GHV}.  We will denote by
$\xi_0$, $\xi_1$ and $\xi_2$ contact structures partitioned by the corresponding
$K_i$. Using the disk whose boundary is dashed, one sees that $K_0$ is
bypass-related to $K_1$ and $K_2$. We denote $(A \times \S^1, F_A \times \S^1)$
by $(N,\Gamma)$.

Let $\xi_b$ be a basic slice on a toric annulus $(N',\Gamma')$. We glue
$(N,\Gamma)$ and $(N',\Gamma')$ to get a new toric annulus. Using the obvious
decomposition of $H_1(N)$ and the corresponding one for $H_1(N\cup N')$, we want
the dividing slopes to be $\infty$ (this is the slope of the $\S^1$ factor) and
$1$ respectively.  By changing the sign of the basic slice, we can assume that
$\xi_0 \cup \xi_b$ is universally tight. It follows from the classification of
tight contact structures on toric annuli that a contact manifold has positive
Giroux torsion if and only if it contains a copy of $\xi_0 \cup \xi_b$.
Therefore we only need to prove that $c(\xi_0 \cup \xi_b)$ vanishes.

Let $\Phi = \Phi_{\xi_b}$ be a corresponding HKM gluing map.  The structures
$\xi_1 \cup \xi_b$ and $\xi_2 \cup \xi_b$ are $\partial$--isotopic and they are
basic slices. Using invariance under isotopy, we get 
$c(\xi_1 \cup \xi_b) = c(\xi_2 \cup \xi_b)$.  Let $\tilde c_b$ be a
representative of this common contact invariant.  Let $\tilde c_1$ and 
$\tilde c_2$ be representatives of $c(K_1)$ and $c(K_2)$ such that 
$\tilde c_b=\Phi(\tilde c_1)=\Phi(\tilde c_1)$.  Such representatives exist 
according to the gluing property. We also take any representative 
$\tilde c(K_0) \in c(K_0)$ and denote by $\tilde c(\xi_0 \cup \xi_b)$ its image
under $\Phi$.  This image belong to $c(\xi_0 \cup \xi_b)$ according to the
gluing property.

Lemma \ref{lem:bypass} gives 
$\varepsilon_1, \varepsilon_2 \in \{\pm 1\}$ such that
\[
\tilde c(K_0) = \varepsilon_1 \tilde c_1 + \varepsilon_2\tilde c_2.
\]

We then apply $\Phi$ to this equation to get:
\begin{equation}
  \label{eqn:c_0_c_GHV}
\tilde c(\xi_0 \cup \xi_b) = (\varepsilon_1 + \varepsilon_2)\tilde c_b.
\end{equation}

Let $(W,\xi_W)$ be a standard neighborhood of a Legendrian knot ($W$ is a solid
torus). We now glue $(W,\xi_W)$  along the boundary component of $N \cup N'$
which is in $\partial N$ so that meridian curves have slope $0$. The structure
$\xi_W \cup \xi_0 \cup \xi_b$ is overtwisted whereas 
$\xi_W \cup \xi_1 \cup \xi_b$ (and $\xi_W \cup \xi_2 \cup \xi_b$ which is
isotopic to it) is a standard neighborhood of a Legendrian curve so can be
embedded into Stein fillable closed contact manifolds.
Let $\Phi_W$ be a gluing map associated to
$\xi_W$. Applying $\Phi_W$ to equation \ref{eqn:c_0_c_GHV} and using the vanishing
property of overtwisted contact structures, we get 
\[
0 = (\varepsilon_1 + \varepsilon_2)\Phi_W(\tilde c_b).
\]
Using that $\Phi_W(\tilde c_b)$ is non zero and the fact that the
relevant $\SFH$ group has no torsion (see \cite{Juhasz_sd}[Proposition 9.1])
we get $\varepsilon_1 + \varepsilon_2 = 0$. Returning
to Equation \ref{eqn:c_0_c_GHV}, we then get $c(\xi_0 \cup \xi_b) = 0$.
\end{proof}

\begin{figure}[ht]
  \begin{center}
  	\includegraphics{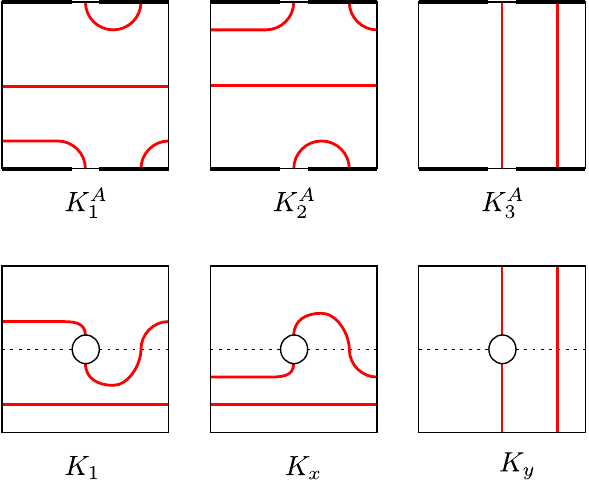}
  \end{center}
  \caption{Dividing sets for Proposition \ref{prop:annulus} and Proposition
  \ref{prop:main}. On the top row, left and right sides of the squares are glued
  to make the annulus $A$. Then the thick parts of $\partial A$ can be glued by
  translation to make the punctured torus of the bottom row where the sides of
	the squares are glued by translation and the glued part of $\partial A$ is
	dashed.}
  \label{fig:div_annulus}
\end{figure}

\begin{prop}
\label{prop:annulus}
Let $(A, F_A)$ be an annulus with two points on each boundary component. 
Let $T$ be one of the components of 
$\partial A \times \S^1$ and $t = e^{[T]} \in \Z[H_2(A \times \S^1)]$. 
Let $K^A_1$, $K^A_2$ and $K^A_3$ be the dividing sets of Figure
\ref{fig:div_annulus} and let
$\tilde\twc^A_1$, $\tilde\twc^A_2$ and $\tilde\twc^A_3$
be any representatives of their contact invariants in $\twV(A,F_A)$.
\begin{enumerate}
\item 
There exist invertible elements $a$ and $b$ in $\LL$ such that:
\[
\tilde\twc^A_1 = a\tilde\twc^A_2 + b(t-1)\tilde\twc^A_3.
\]

\item
Twisted invariants distinguish $K^A_1$, $K^A_2$ and $K^A_3$. Over $\Z$
coefficients,  $c(K^A_2)$ and $c(K^A_3)$ are independent but
$c(K^A_1) = c(K^A_2)$. 

\item
Let $\tau$ be the right handed Dehn twist along the core of $A$. There exist
$\tilde c_2 \in c(K^A_2)$ and $\tilde c_3 \in c(K^A_3)$ such that for any
$n \in \Z$, $\tilde c_3 + n \tilde c_2 \in c(\tau^n K^A_3)$.
\end{enumerate}
\end{prop}

The second part of this proposition was proved \over $\Z$ coefficients in
Section 7.5 of \cite{HKM_tqft}. The last part was conjectured in
\cite{HKM_tqft}[top of page 35].

\begin{proof}
The statement of the proof contains $A$ superscripts everywhere in view of its
application to Proposition \ref{prop:main} but we don't use them in this proof
since it would clutter all formulas.

Thanks to grading, the twisted invariants $\tilde\twc_1$, $\tilde\twc_2$ and
$\tilde\twc_3$ all live in the same rank two summand of $\twV(A, F_A)$ so
there exist $\lambda, \mu, \nu \in \LL$, not all zero, such that 
\begin{equation}
  \label{eqn:ann1}
\lambda \tilde\twc_1 = \mu \tilde\twc_2 + \nu \tilde\twc_3.
\end{equation}

We now use two HKM gluing maps: $\twPhi_1$ (resp. $\twPhi_2$) corresponding to
gluing the dividing set $K_1$ (resp. $K_2$) from the bottom in Figure
\ref{fig:div_annulus}. We will denote loosely by $K_1 \cup K_1$ for instance
the result of gluing $K_1$ on the bottom of $K_1$. For any $\xi$ in
partitioned by $K_1$ we can perform a generalized Lutz twist on the unique
torus which is foliated by Legendrian fibers and the result is partitioned by
$K_1 \cup K_1$ so the main result of \cite{GH} gives
$\twPhi_1(\tilde\twc_1) = d(t-1)\tilde\twc_1$
for some invertible element $d$.  Since contact structures partitioned by $K_1
\cup K_2$ are overtwisted, we get $\twPhi_1(\tilde\twc_2) = 0$. And 
$K_1 \cup K_3$ is isotopic to $K_1$ so there is some invertible $e$ such that 
$\twPhi_1(\tilde\twc_3) = e\tilde\twc_1$. So when we apply $\twPhi_1$ to
equation \ref{eqn:ann1} we get: 
$\lambda d(t-1)\tilde\twc_1 = \nu e\tilde\twc_1$.

A similar argument for $\twPhi_2$ gives invertible elements $f$ and $g$ such that:
\[
\mu f(t-1)\tilde\twc_2 + \nu g\tilde\twc_2 = 0.
\]
Since $\twSFH(A, F_A)$ is a free module over the integral domain $\LL$ and 
$\tilde\twc_1$ and $\tilde\twc_2$ are non zero (the corresponding contact
structures embed into Stein fillable contact manifolds), we get
\begin{gather*}
\lambda d(t-1) = \nu e \\
\mu f(t-1) + \nu g = 0.
\end{gather*}
so that $\nu = \lambda e^{-1} d(t-1)$ and $\mu = -f^{-1}ge^{-1}d\lambda$. Since
$\lambda$, $\mu$ and $\nu$ are not all zero, we get that $\lambda$ is non zero.
Setting $a = -f^{-1}ge^{-1}d$ and $b = e^{-1} d$, equation \ref{eqn:ann1} gives
the announced relation.

We now prove the second point. We have already met morphisms sending
$\tilde\twc_1$, $\tilde\twc_2$ and $\tilde\twc_3$ to elements not related to
each other by invertible elements of $\LL$. So the invariants $\twc(K_i)$ are
pairwise distinct. Going to $\Z$ coefficients sends $t - 1$ to zero so the
formula of the first point proves that \OS invariants \over $\Z$ coefficients
don't distinguish $K_1$ and $K_2$. But they distinguish $K_1$ and $K_3$ as can
be seen for instance by using the $\Z$ coefficients version of $\twPhi_1$.

In order to prove the third point we will use the results of Section \ref{S:t3}.

Figure \ref{fig:div_GHV} shows that $K_2$, $K_3$ and $\tau^{-1}K_3$ are
bypass related. We start with any representatives for the relevant invariants
and Lemma \ref{lem:bypass} gives us instructions to change signs so that we get 
$\tilde c_2 \in c(K_2)$ and $\tilde c_3 \in c(K_3)$ with
$\tilde c_3 -  \tilde c_2 \in c(\tau^{-1} K_3)$.

We now stick to these representatives. Using the image of Figure
\ref{fig:div_GHV} under $\tau$, we see that $K_2$, $K_3$ and $\tau K_3$
are bypass related. So Lemma \ref{lem:bypass} gives signs $\varepsilon$ and
$\varepsilon'$ such that $\varepsilon \tilde c_2 + \varepsilon'\tilde c_3$ is in
$c(\tau K_3)$. We set $\varepsilon_1 = \varepsilon\varepsilon'$ so that
$\tilde c_2 + \varepsilon_1\tilde c_3$ is in $c(\tau K_3)$. We want to prove
that $\varepsilon_1 = 1$. the only other possibility, $\varepsilon_1 = -1$ would
give $c(\tau^{-1} K_3) = c(\tau K_3)$ but this is forbidden by Theorem
\ref{thm:t3} since the corresponding contact structures are sent by gluing the
two components of $A$ to contact structures on $\T^3$ which are distinguished by
\OS invariants. So $\tilde c_2 + \tilde c_3$ is in $c(\tau K_3)$. The general
case follows from an inductive process using the same arguments.
\end{proof}

\begin{figure}[ht]
  \begin{center}
  	\includegraphics{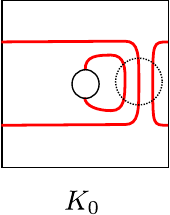}
  \end{center}
  \caption{Dividing set for Proposition \ref{prop:main}. Sides of the square are
  		   glued to make a punctured torus. The dashed circle bounds a disk used
		   to apply Lemma \ref{lem:bypass}}
  \label{fig:K_0}
\end{figure}

\begin{prop}
  \label{prop:main}
Let $\Sigma_0$ be a punctured torus, $F_0$ a set of two points on
$\partial\Sigma_0$ and $K_0$ a dividing set on $\Sigma_0$ consisting of a
circle and an arc, both boundary--parallel (see Figure \ref{fig:K_0}). Let 
$(x, y)$ be the image in $\LL = \Z[H_2(\Sigma_0 \times \S^1)]$ of a basis of 
$H_2(\Sigma_0 \times \S^1)$. Let $K_x$ and $K_y$ be dividing sets on $\Sigma_0$
made of a boundary parallel arc and one closed curve whose lift in 
$H_2(\Sigma_0 \times \S^1)$ has homology class $x^{\pm 1}$ and $y^{\pm 1}$
respectively, see Figure \ref{fig:div_annulus}. 
Let $\tilde \twc_0$, $\tilde \twc_x$ and $\tilde \twc_y$ be
any representatives of $\twc(K_0)$, $\twc(K_x)$ and $\twc(K_y)$ respectively.
In $V(\Sigma_0, F_0)$, \[c(K_0) = 0\] whereas there exist invertible elements
$\lambda$, $\mu$ in $\LL$ such that, in $\twV(\Sigma_0, F_0)$,
\[\tilde\twc_0 = \lambda(y - 1)\tilde\twc_x + \mu(x - 1)\tilde\twc_y.\]
\end{prop}

Before proving this proposition we discuss its application to Theorem
\ref{thm:leg} from the introduction.  Let $T$ be a torus obtained by filling the
boundary of $\Sigma_0$ with a disk $D$. Let $a$ be an arc in $D$ with boundary
$F_0$ which extends smoothly $K_0$ to a closed multicurve $\bar K_0$ in $T$. Let
$V$ be a circle bundle over $T$ with Euler number $\pm 1$. There is an
overtwisted contact structure $\xi$ on $V$ partitioned by $\bar K_0$ and the
fiber over any point of $a$ is a null homologous Legendrian knot. The
restriction of $\xi$ to the solid torus over $D$ is a standard Legendrian
neighborhood of $L$ according to the easiest case of the classification of tight
contact structures on solid tori. So $c(K_0)$ can be seen as the sutured
invariant of the Legendrian knot $L$ and we proved Theorem \ref{thm:leg}. Many
more examples of this situation can be constructed using Theorem \ref{thm:tqft}
above.

In the above Proposition \ref{prop:main}, the formula for the twisted invariant
clearly implies vanishing of the untwisted invariant but, for the benefit of
readers which are not interested in twisted coefficients, we will explain how to
get directly the vanishing result.

\begin{proof}
Using the disk whose boundary is dashed on Figure \ref{fig:K_0}, one sees 
that $K_0$ is bypass-related to $K_1$ and $K_x$ from Figure \ref{fig:div_annulus}.

The dividing sets $K_1$ and $K_x$ are obtained from the dividing sets of
Proposition \ref{prop:annulus} as explained in Figure \ref{fig:div_annulus}. 
Let $\Phi_A$ be
a HKM gluing map associated to the thick annuli of this figure, glued by
translation. Let $\tilde\twc^A_i$, $i = 1, 2, 3$ be representatives of the
$\twc(K^A_i)$. We know there are invertible elements $f$, $g$, $h$ in $\LL$ such that
$\Phi_A(\tilde\twc^A_1) = f^{-1}\tilde \twc_1$,
$\Phi_A(\tilde\twc^A_2) = g\tilde \twc_x$ and
$\Phi_A(\tilde\twc^A_3) = h\tilde \twc_y$. 

Lemma \ref{lem:bypass} gives 
$d, e \in \Lx$ such that
\begin{equation}
  \label{eqn:c_0_c}
\tilde \twc_0 = d \tilde\twc_1 + e \tilde\twc_x.
\end{equation}

We then apply $\Phi_A$ to the equation of Proposition \ref{prop:annulus} to get
\begin{equation}
  \label{eqn:c_0_cA}
	\tilde \twc_0 = (r + e)\tilde\twc_x + \mu(x - 1)\tilde\twc_y.
\end{equation}
where $r = dfag$ and $\mu =  dfhb$ are invertible.

We first prove quickly vanishing of the untwisted invariant and then we'll turn
again to twisted coefficients. Over $\Z$ coefficients, the preceding equation
reduces to 
\begin{equation}
  \label{eqn:c_0_cA_Z}
	\tilde \twc_0 = (\varepsilon_1 + \varepsilon_2)\tilde\twc_x
\end{equation}
where $\varepsilon_1$ and $\varepsilon_2$ are signs.

Let $D$ be a disk divided by an arc $K_D$ and $\xi_D$ a 
contact structure on $D \times \S^1$ partitioned by $K_D$. 
We now glue $(\Sigma_0, K_0)$ to $(D, K_D)$ and consider a
HKM gluing map $\Phi : V(\Sigma_0, F_0) \to \HFhat(\T^3)$ given by $\xi_D$.

According to Giroux's criterion (contained in the first part of Theorem
\ref{thm:classif_partition}), $\xi_0 \cup \xi_D$ is overtwisted.
Since overtwisted contact structures have vanishing invariant, we get 
$\Phi(\tilde c_0) \in c(\xi_0 \cup \xi_D) = 0$. So equation \ref{eqn:c_0_cA_Z} gives
$0 = (\varepsilon_1 + \varepsilon_2) \Phi(\tilde c_x)$. In addition
$\Phi(\tilde c_x) \in c(\xi_1 \cup \xi_D)$ is non zero because 
$\xi_1 \cup \xi_D$ is Stein fillable. Since $\HFhat(\T^3)$ has no torsion (see
Section \ref{S:t3}), we get
that $\varepsilon_1 + \varepsilon_2 = 0$ and $\tilde c_0 = 0$ so $c(K_0) = 0$. 

We now return to twisted coefficients. We glue in an annulus divided by two
boundary parallel arcs, see Figure \ref{fig:strange_formula}.
\begin{figure}[ht]
  \begin{center}
  	\includegraphics{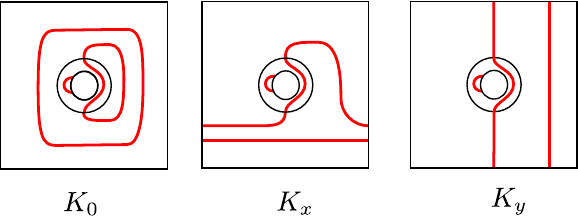}
  \end{center}
	\caption{A gluing for Proposition \ref{prop:main}. Opposite edges of each
	square are glued to get punctured tori.}
  \label{fig:strange_formula}
\end{figure}
When glued to $K_0$ we get an overtwisted contact structure while $K_x$ and
$K_y$ lead to generalized Lutz modifications on the same dividing set $K$. Let
$\tilde\twc$ be a representative of $c(K)$ (which is not zero since its $\Z$
coefficient projection does not vanish).
Using the Ghiggini--Honda formula \cite{GH}, we get invertible elements $u$ and $v$ in
$\LL$ such that 
\[
	0 = (x - 1)u(r + e)\tilde\twc + (y - 1)v\mu(x - 1)\tilde\twc.
\]
We can now use that $\twV$ is a free $\LL$ module (Proposition \ref{prop:twV})
and $\LL$ is an integral domain to get $r + e = -u^{-1}v\mu(y - 1)$ so that
Equation \eqref{eqn:c_0_cA} gives the expected formula with 
$\lambda = -u^{-1}v\mu$.
\end{proof}

Now that we have Proposition \ref{prop:main}, the following proof is almost
identical to that of \cite{HKM_tqft}[Proposition 7.12]

\begin{proof}[Proof of Theorem \ref{thm:tqft}]
First remark that $c(K) = 0$ if $K$ has an isolated annulus  because the
corresponding contact structures have non zero Giroux torsion. Then we will use
two nested inductive proofs to get the general result.

We now start an induction on the number of boundary components of
isolated regions. 

First assume that $K$ has an isolated region $\Sigma_0$ whose boundary 
is connected. We prove the theorem by
induction on the genus of $\Sigma_0$. If this genus is zero then any contact
structure partitioned by $K$ is overtwisted hence $c(K) = 0$.
If this genus is one then $\Sigma_0$ is a
punctured torus and $\Sigma$ contains a sub-surface satisfying the assumptions
of Proposition \ref{prop:main} so, by this proposition and Theorem
\ref{thm:ci}, $c(K) = 0$. Assume now that the theorem is
proved when $K$ has an isolated region with connected boundary and genus at most
$g-1 \geq 1$. If $K$ has an isolated region with genus $g > 1$ then $(\Sigma, K)$ has
a subsurface $(\Sigma_1, K_1)$ drawn on the left-hand side of Figure
\ref{fig:recursion} where the sides of the square are glued pairwise and the
shaded disk hides a subsurface having genus $g-1$ and not intersecting $K$. The
dashed circle shows that $K_1$ is bypass-related to $K_1$ and $K_2$. Since $K_2$
has an isolated punctured torus and $K_3$ has an isolated region with genus 
$g - 1$, the inductive hypothesis gives $c(K_2) = c(K_3) = 0$. Lemma
\ref{lem:bypass} combines these two vanishings to give $c(K_1) = 0$. This
implies $c(K) = 0$ thanks to Theorem \ref{thm:ci}. Hence the inductive step is
completed and any $K$ having an isolated region with connected boundary has
vanishing invariant.

\begin{figure}[ht]
  \begin{center}
  	\includegraphics{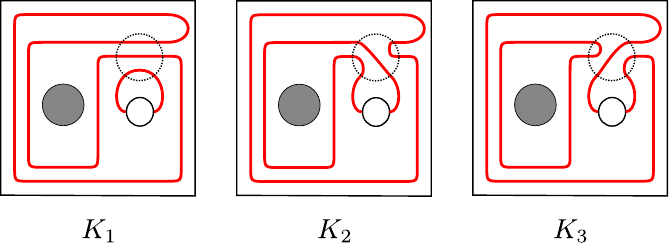}
  \end{center}
  \caption{First inductive step in the proof of Theorem \ref{thm:tqft}. The sides of
  the squares are glued by translation and the shaded disk hides more genus.}
  \label{fig:recursion}
\end{figure}

\begin{figure}[ht]
  \begin{center}
  	\includegraphics{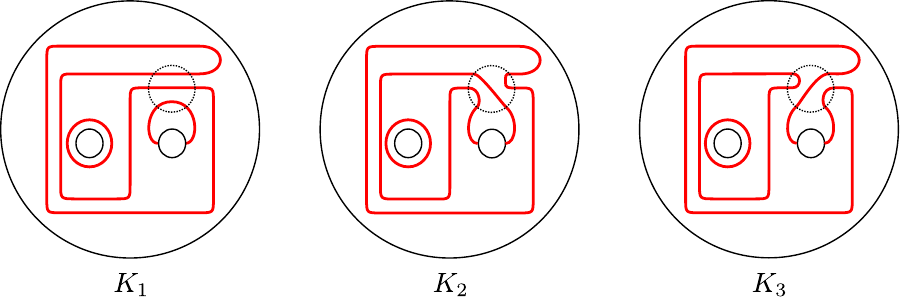}
  \end{center}
  \caption{Second inductive step in the proof of Theorem \ref{thm:tqft}.}
  \label{fig:recursion2}
\end{figure}

We now prove the induction step for our original inductive proof. 
We assume the theorem is proved for any dividing set having an
isolated region with at most $r-1 \geq 1$ boundary components. Suppose
$K$ has an isolated region
$\Sigma_0$ with $r > 1$ boundary components. We can assume that $\Sigma_0$ is
not an annulus since this case is already known. Also, at least one boundary 
component $\gamma$ of $\Sigma_0$ is adjacent to another region whose closure
meets $K \setminus \gamma$. Then $\Sigma$ has a subsurface $\Sigma'$ which is a
twice punctured disk whose intersection with $K$ is $K_1$ shown on Figure
\ref{fig:recursion2}. This figure also shows the intersections  with $\Sigma'$
of dividing sets $K_2$ and $K_3$ which are bypass-related to $K$. The dividing
set $K_2$ has an isolated region with $r-1$ boundary components. One of them is
the outermost thick circle of Figure \ref{fig:recursion2}, the other ones are
not in $\Sigma'$. So $c(K_2) = 0$ by inductive assumption.
The dividing set $K_3$ has an annular isolated region so $c(K_3) = 0$.
Lemma \ref{lem:bypass} combines these two vanishings to
give $c(K) = 0$.
\end{proof}

\begin{proof}[Proof of the main theorem]
Let $V$ be a Seifert manifold over an orbifold $B$ whose base has genus at least
three. Let $K_0$ be the multi-curve of Figure \ref{fig:main_example} where $B$
continues to the right and all exceptional points of $B$ are in the right
hand side of the picture. Let $\tau$ be the right-handed Dehn twist around the
thick (black) curve of Figure \ref{fig:main_example}. Theorem
\ref{thm:classif_partition} associates to the $\tau^n(K_0)$'s infinitely many
isotopy classes of universally tight torsion free contact structures. Note that
the genus hypothesis is used here to ensure that our contact structures are
torsion free.  Proposition \ref{prop:main} and Theorem \ref{thm:ci} ensure that
they all have vanishing contact invariant \over $\Z$ coefficients since there
dividing sets all contain a copy the dividing set of Proposition
\ref{prop:main}.
\end{proof}

\begin{figure}[ht]
  \begin{center}
  	\includegraphics[scale=.6]{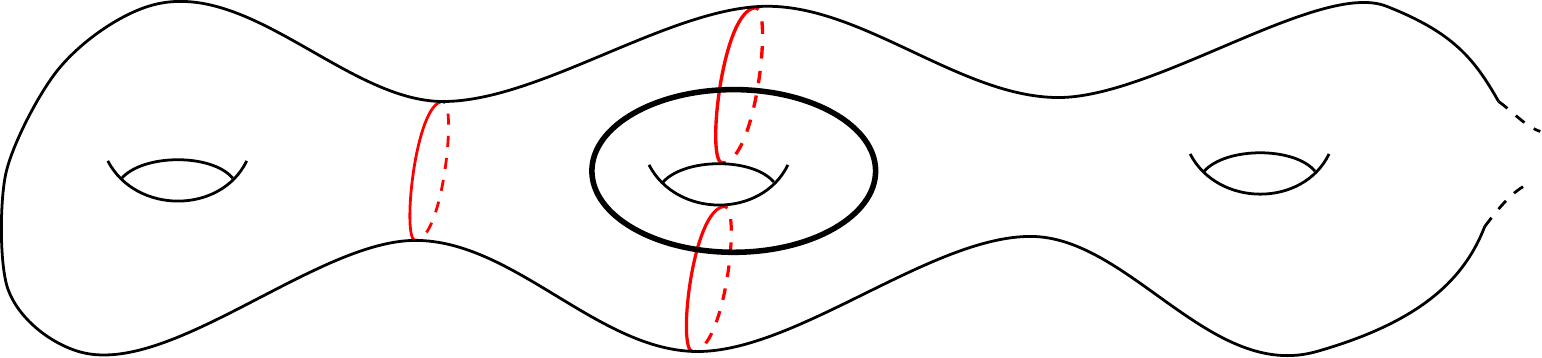}
  \end{center}
  \caption{Universally tight torsion free contact structures with vanishing
  contact invariants. The partitioning curve is thin (and red)}
  \label{fig:main_example}
\end{figure}

\paragraph{Acknowledgements.} I learned that the existence of torsion free
contact structures with vanishing \OS invariants was open during the Workshop on
Symplectic Geometry, Contact Geometry and Interactions in Strasbourg in January
2009. I thank the organizers of this event. I benefited from conversations about
Heegaard--Floer homology with András Juhász, Tom Mrowka, Peter Ozsváth, András
Stipsicz and a lot of conversations with Paolo Ghiggini. This project was also
stimulated by conversations with Chris Wendl who works on the same questions on
the embedded contact homology side of the story. In addition, Chris Wendl and
Paolo Ghiggini helped me to find an error in a draft of this text.

This work was partially supported by ANR grants ``Symplexe'' and ``Floer
Power''.

\bibliographystyle{alpha}
\small

\bibliography{sfh}

\vspace{.3cm}
\noindent
\textsc{École Normale Supérieure de Lyon, 69000 Lyon, France}

\vspace{.3cm}
\noindent
Current adress:

\noindent
\textsc{Université Paris Sud 11, 91430 Orsay, France}
\end{document}